\documentclass[12pt]{amsart}
\usepackage{amsfonts}
\usepackage{amsmath}
\usepackage{amssymb}
\usepackage[margin=1.2in]{geometry}
\usepackage{verbatim}
\usepackage{amsthm}
\usepackage{stmaryrd}

\usepackage{paralist}

\usepackage{xcolor}
\usepackage[hidelinks]{hyperref}
\hypersetup{
    colorlinks = true,
    linkcolor = red,
    anchorcolor = red,
    citecolor = green,
    filecolor = red,
    urlcolor = red,
    linktoc=page
    }
\usepackage{cleveref}
\usepackage{tikz}
\usetikzlibrary{shapes.geometric,decorations.pathmorphing,patterns}
\usepackage[font=small,labelfont=bf]{caption}
\usepackage{graphicx}
\graphicspath{{./figures/induction_heating/}}
\usepackage{subcaption}
\usepackage{comment}

\makeatletter
\def\namedlabel#1#2{\begingroup
   \def\@currentlabel{#2}%
   \label{#1}\endgroup
}
\makeatother

\usepackage{enumerate}

\newtheorem{theorem}{Theorem}[section]
\newtheorem{corollary}{Corollary}[section]
\newtheorem{lemma}{Lemma}[section]

\theoremstyle{remark}
\newtheorem{remark}[theorem]{Remark}
\newtheorem{example}{Example}

\newcommand {\NN } {{\mathbb N}}
\newcommand {\RR } {{\mathbb R}}

\def \vector#1{\mathbf{#1}}

\newcommand {\X}{\vector{x}}

\newcommand {\I}{[0,T]}
\newcommand {\Iopen}{(0,T)}

\def\half{{\frac{1}{2}}}

\newcommand {\domain}{\Theta}

\newcommand {\pdt}{\partial_t}
\newcommand {\pdtt}{\partial_{tt}}

\DeclareMathOperator{\Leb}{L}
\DeclareMathOperator{\Lip}{Lip}
\DeclareMathOperator{\Cont}{C}
\DeclareMathOperator{\Hi}{H}

\DeclareMathOperator{\di}{d\hspace{-1.5pt}}
\DeclareMathOperator{\didiff}{d}

\newcommand {\dx}{ \di x}

\newcommand {\dt}{ \di t}
\newcommand {\ds}{\ \di s}
\newcommand {\dr}{\ \di r}

\newcommand {\dX}{\ \di \X}

\newcommand {\scal}[2]{\left(#1,#2\right)}

\newcommand{\abs}[1]{\left\lvert #1 \right\rvert}
\newcommand{\vnorma}[1]{\left\|#1\right\|}

\newcommand {\imbed} {\hookrightarrow}

\newcommand {\lp}[1]{\Leb^{#1} (\domain )}

\newcommand {\Lp}[1]{{\bf L}^{#1} (\domain )}

\newcommand {\hk}[1]{\Hi^{#1}(\domain )}
\newcommand {\hko}[1]{\Hi^{#1}_0(\domain )}

\newcommand {\lpkIX}[2]{\Leb^{#1}\left(\Iopen,#2\right)}

\newcommand {\cIX}[1]{\Cont\left(\I,#1\right)}

\makeatletter
\@namedef{subjclassname@2020}{%
  \textup{2020} Mathematics Subject Classification}
\makeatother

\begin{document}

	\title[Nonlinear wave equation with variable-order damping]{On the Rothe-Galerkin spectral discretisation for a class of variable fractional-order nonlinear wave equations}

	\author[K. Van Bockstal]{Karel Van Bockstal$^1$} 
	\thanks{The work of K.~Van Bockstal was supported by the Methusalem programme of Ghent University Special Research Fund (BOF) (Grant Number 01M01021).} 

\author[M.~S. A. Zaky]{Mahmoud A. Zaky$^{2,3}$}

 \author[A.~S. Hendy]{Ahmed S. Hendy$^{4,5}$}
\thanks{A. S. Hendy wishes to acknowledge the support of the RSF, Russia grant, project 22-21-00075. }

\address[1]{Ghent Analysis \& PDE center, Department of Mathematics: Analysis, Logic and Discrete Mathematics, Ghent University, Krijgslaan 281, 9000 Ghent, Belgium}

\email{karel.vanbockstal@ugent.be}

\address[2]{Department of Mathematics and Statistics, College of Science, Imam Mohammad Ibn Saud Islamic University, Riyadh, Saudi Arabia}
\email{ma.zaky@yahoo.com}
\address[3]{Department of Applied Mathematics, National Research Centre, Dokki, Giza 12622, Egypt}

	\address[4]{Department of Computational Mathematics and Computer Science, Institute of Natural Sciences and Mathematics, Ural Federal University, 19 Mira St., Yekaterinburg 620002, Russia}
 	\email{ahmed.hendy@fsc.bu.edu.eg}
 \address[5]{Department of Mathematics, Faculty of Science, Benha University, Benha 13511, Egypt}

	
	\begin{abstract} 
In this contribution, a wave equation with a time-dependent variable-order fractional damping term and a nonlinear source is considered. Avoiding the circumstances of expressing the nonlinear variable-order fractional wave equations via closed-form expressions in terms of special functions, we investigate the existence and uniqueness of this problem with Rothe's method. First, the weak formulation for the considered wave problem is proposed.  Then, the uniqueness of a solution is established by employing Gr\"onwall's lemma. The Rothe scheme's basic idea is to use Rothe functions to extend the solutions on single-time steps over the entire time frame. Inspired by that, we next introduce a uniform mesh time-discrete scheme based on a discrete convolution approximation in the backward sense. By applying some reasonable assumptions to the given data, we can predict a priori estimates for the time-discrete solution. Employing these estimates side by side with Rothe functions leads to proof of the solution's existence over the whole time interval. Finally, the full discretisation of the problem is introduced by invoking Galerkin spectral techniques in the spatial direction, and numerical examples are given.
	\end{abstract}
	
	\maketitle
	
	\tableofcontents

\section{Introduction}
\label{sec:intro} 

\subsection{Formulation of the problem}

Let $\Theta \subset \RR^d$ be a bounded Lipschitz domain with  $\partial \Theta$ as the boundary. The variable-order fractional (V-OF) integral  ${}_0I_t^{\mu (t)} $, and the V-OF Caputo derivative $\frac {\partial ^{\mu(t)} } {\partial t ^{\mu(t)}}$ are defined, respectively, as \cite{lorenzo2002variable,zhuang2009numerical}
\begin{equation*}
{}_0I_t^{\mu (t)} \Phi(t): = \frac{1}{{\Gamma \left( {\mu (t)} \right)}}\int_0^t {\frac{{\Phi(r)}}{{(t - r)^{1 - \mu (t)} }}\dr} ,
\end{equation*}
\begin{equation}\label{eq:caputo_var_order}
\partial^{\mu(t)}_t \Phi(t):= {}_0I_t^{1 - \mu (t)} \Phi^\prime(t) = \frac{1}{{\Gamma\left( 1 -  {\mu (t)} \right)}}\int_0^t {\frac{{\Phi^\prime(r)}}{{(t - r)^{\mu (t)} }}\dr}.
\end{equation}
We consider $\mu \in \Leb^\infty(\I)$ such that 
\[
0 \leqslant \mu(t) \leqslant \bar \mu := \sup_{t\in\I} \mu(t) <1. 
\]
In this paper, we consider the following V-OF nonlinear problem   
  \begin{multline}
  \label{eq:problem}%
    \frac {\partial^2  \Phi} {\partial t^2} (\X,t) + \rho(\X,t)  \partial^{\mu(t)}_t \Phi (\X,t) \\
    = \nabla \cdot \left[\beta(\X,t) \nabla \Phi(\X,t) \right] +f(\Phi (\X , t))+Q(\X,t),\quad   (t,\X) \in   (0,T]\times \Theta,
\end{multline}
  with the initial-boundary conditions of the form
\begin{equation}
  \label{eq:ics_bcs}%
  \left\{ \begin{array}{rlr}
    \Phi (\X , 0) &= \tilde \Phi_0(\X), & \X  \in \Theta, \\
   \pdt  \Phi (\X , 0)& = \tilde\Psi_0(\X), & \X  \in \Theta,\\ 
    \Phi &=  0, &  \text{ on } \partial \Theta \times (0,T]. 
  \end{array} \right.
\end{equation}

\subsection{Literature}

%\cite{Zheng2020}
%\cite{Zheng2021,Zheng2022}

One of the mean features of {{\color{black} constant order} Caputo time fractional subdiffusion equations and their applications is the initial weak singularity \cite{sakamoto2011initial,luchko2012initial}. {\color{black} The initial weak singularity also appears in its the time-fractional wave equation analogue \cite{Otarola2019}.
In this contribution, a wave equation with a time-dependent variable-order fractional damping term is considered. The solution to this problem will not exhibit this behaviour. 
}

 One possible inspiration for the model \eqref{eq:problem}-\eqref{eq:ics_bcs} is the vibration of a perfectly elastic membrane {\color{black} in viscoelastic media.}  As the equilibrium state of the membrane, it is tautly stretched along the boundary of the $\partial \Theta$  of the physical domain in the plane $\Theta$, and a $\Phi$  axis is constructed perpendicular to the $\X$  plane. {\color{black} More details about the reasoning for considering the variable-order fractional derivative as a damping term in this setting can be found in \cite[Section~2]{Zheng2022}. A glimpse of V-OF derivatives was initially given by Samko and Ross in 1993 \cite{samko1993fractional}. } For more explanations about the applications of V-OF  problems in mechanics, viscoelasticity, transport processes, and control theory, we refer to \cite{patnaik2020applications}. {\color{black}  Moreover, Sun et al. \cite{sun2019review} provided a survey of the recent literature and findings, including fundamental definitions,  numerical methods,  models, and their applications for V-OF differential equations. }

Theoretical findings on the well-posedness and regularity of V-OF problems are limited. {\color{black} A reason can be the impossibility of analysing through giving the analytical formulation in terms of special functions, which happens more easily in their constant-order counterparts.
} 
Concerning the space-dependent V-OF, we mention \cite{Kian2018,van2021existence,VanBockstal_2022_QM}. In \cite{Kian2018}, the existence of a unique weak solution (in the sense of the Laplace original) to the V-OF diffusion equation has been studied. The governing elliptic operator is supposed to be autonomous. Van Bockstal investigated a similar problem in \cite{van2021existence}. However, the coefficients accompanying the problem have a temporal and spatial variable dependency. Armed with the strongly positive definiteness of the governing kernel and the assumption of the belonging of the initial data to $\Hi_0^1 (\Theta)$, the existence of a unique weak solution has been derived. An analogue result for the space-dependent V-OF wave equation has been obtained in \cite{VanBockstal_2022_QM}.
Concerning the time-dependent V-OF, we state the contributions \cite{Wang2019b,Zheng2020,Zheng2021,Zheng2022,van2022existence}. In \cite{Wang2019b}, the analogue of \eqref{eq:problem} has been considered for the linear diffusion model but with a V-OF Riemann-Liouville fractional derivative (with $\mu \in \Cont(\I)$ satisfying $0 < \mu_m \leqslant \mu(t)\leqslant 1$ and $\lim_{t \searrow 0} (\mu(t)-\mu(0))\ln(t) =0$) and solely space-dependent coefficients. The authors establish the well-posedness of the problem in multiple space dimensions (using eigenfunction expansion on a smooth domain $\Theta$) and show that the regularity of the solution depends on the value of $\mu(0)$. Then, in \cite{Zheng2020}, the authors study the well-posedness of the diffusion problem with current-stated based V-OF operator \eqref{eq:caputo_var_order}. Also, in this situation, the regularity depends on the value of $\mu(0).$ In \cite{Zheng2021}, the authors study problem \eqref{eq:problem} for $\beta (\X,t) = K>0$ and $f=0$. Again using spectral decomposition, the authors prove the well-posedness of the problem if $\mu \in \Cont^1(\I)$ satisfies $0<\mu(t)<1.$ A full discretisation of this problem has been studied in \cite{Zheng2022}. Finally, in \cite{van2022existence}, we have investigated the existence and uniqueness of a history-state time-delay V-OF diffusion equation with damping subjected to weak assumptions on the data. A derivation of a priori estimates is given, and so the existence of the weak solution to the considered problem has been established on a specific time frame $\left[0, \displaystyle \lfloor \frac{T}{s} \rfloor s\right]$, such that $s>0$ is a fixed positive delay parameter conditioned by $s\leqslant T,$ where $T$ is the end time.\\

The new aspect of our contribution is that we formulate the problem on a more general domain and consider a time-dependent diffusion coefficient and nonlinear source term. Consequently, the spectral decomposition approach is not applicable, and for this reason, we tackle the problem with the aid of Rothe's method.  The method was originally constructed in \cite{rektorys1987some,Kacur1985} as a discretisation technique for partial differential equations. It  was introduced in \cite{Kacur1985} as an accurate theoretical tool for solving a wide range of evolution problems. Note that Rothe's methods have been also successfully applied when solving fractional evolution problems of constant order \cite{VanBockstal2020b,VanBockstal2020c}, and inverse problems \cite{hendy2022reconstruction,hendy2021solely}.  The main advantage of our approach is that we do not require $\mu \in \Cont^1(\I)$ as in \cite{Zheng2021,Zheng2022} but only $\mu \in \Leb^\infty(\I).$

\subsection{Aims and outline}

  We are giving a deep look at the existence and uniqueness of a weak solution to \eqref{eq:problem}-\eqref{eq:ics_bcs} by invoking Rothe's method, which aids in prolonging the
solutions on the single time steps on the whole time frame. Our analysis will be based on some reasonable assumptions on the given data (AS1)-(AS8). The importance of these assumptions comes from their invocation in the main results of this paper as in the uniqueness Theorem~\ref{thm:uniqueness1}, the auxiliary Lemmas~\ref{lemnn}, \ref{lem:est1} and the existence Theorem~\ref{thm:existence}. The weak formulation of \eqref{eq:problem}-\eqref{eq:ics_bcs} is introduced in Section~\ref{sec:weak_form} paying attention to the possibility of the above boundedness for the $\Leb^2$ norm of time Caputo V-OF derivative as in \eqref{eq:useful_estimate2} in the later context. Also, the regularity of the solution is given in Section~\ref{sec:weak_form} for a special case when $\mu(t)=\mu$ by employing Fourier analysis. Section~\ref{sec:weak_form} ends with a discussion of the uniqueness of the solution under an additional assumption. An equidistant partitioning of the time interval and a uniform mesh approximation of the V-OF Caputo operator is used to give the semi-discretisation of \eqref{eq:problem}-\eqref{eq:ics_bcs}. It is introduced in Section~\ref{sec:time_discretization} side by side with deriving a priori estimates for the time-discretised solution. With the aid of these estimates, which are given on single time steps, we study the existence of the solution in Section~\ref{sec:existence} by prolonging the discretised solutions on the whole time interval by Rothe functions. Before closing, we give some numerical experiments in Section~\ref{sec: numerical} by fully discretising the problem under consideration in space using Galerkin Legendre spectral polynomials.    

%%%%%%%%%%%%%%%%%%%%%%%%%%%%%

\section{Weak formulation} 
\label{sec:weak_form}

In our analysis, we put the following assumptions (AS) on the given data 
\begin{itemize}
    \item  (AS1):  $\mu\in \Leb^\infty(\I)$ with  $0 \leqslant \mu(t)\leqslant\bar \mu <1$; 
\item (AS2): $\rho,\beta \in {\mathcal X} := \Leb^\infty\left(\Theta \times \Iopen\right)$;
\item (AS3): $Q\in  \cIX{\lp{2}}$;
\item (AS4): $f:\mathbb{R}\to \mathbb{R}$ is Lipschitz continuous, i.e. 
there exists a strictly positive constant $L_f$ such that 
\[
    \abs{f(z_1) - f(z_2)}\leqslant L_f \abs{z_1-z_2};
\]
\item (AS5): $\beta \geqslant \tilde{\beta}_0$ a.e. in $\Theta \times \Iopen$;
\item (AS6): $\pdt\beta \in {\mathcal X}$;
\item (AS7): $\rho \geqslant 0$ a.e. in $\Theta \times \Iopen$; 
\item (AS8): $\tilde{\Phi}_0 \in \hko{1}$ and $\tilde\Psi_0 \in \lp{2}$. 
\end{itemize}
First, we note that
\begin{equation}\label{eq:useful_estimate}
\max_{r\in [0,t]} r^{\bar{\mu}-\mu (t)} \leqslant \max\{T,1\}, \quad \forall t \in \I,
\end{equation}
as $\bar{\mu}-\mu (t) \in [0,\bar{\mu}]\subset [0,1).$ 
Using $ 1= \Gamma(1) \leqslant \Gamma(x)$ for all $x\in(0,1)$, we see that 
\begin{multline*}
   \int_\Theta \int_0^T \abs{\rho(\X,t) \partial^{\mu(t)}_t \Phi(\X,t)}^2 \dt \dX \\
\overset{\eqref{eq:useful_estimate}}{\leqslant} \max\{T,1\}^2 \vnorma{\rho}^2_{{\mathcal X}} \int_\Theta \int_0^T \left(g \ast \abs{\pdt \Phi(\X)} \right)^2(t) \dt \dX, 
\end{multline*}
where 
\begin{equation*}\label{eq:def_kernel}
g(t):= t^{-\bar{\mu}}
\end{equation*}
and `$\ast$' denotes the Laplace convolution defined by 
\[
(u\ast m)(t) = \int_0^t u(t-s)m(s)\ds.
\]
It is clear that $g\in \Leb^1\Iopen$ since $\vnorma{g}_{\Leb^1\Iopen} = T^{1-\bar{\mu}}$.
Now, we will use Young's inequality for convolutions: %Bogachev, Hewitt
\begin{equation}\label{eq:young_conv}
\vnorma{w_1 \ast w_2}_{\Leb^r\Iopen} \leqslant \vnorma{w_1}_{\Leb^p\Iopen} \vnorma{w_2}_{\Leb^q\Iopen},  
\end{equation}
where $\frac{1}{p} + \frac{1}{q} = \frac{1}{r}+1$ with $1 \leqslant p,q \leqslant r \leqslant \infty$. 
This fundamental inequality implies that 
\begin{multline*}
  \int_\Theta \int_0^T \abs{\rho(\X,t)\partial^{\mu(t)}_t \Phi(\X,t)}^2 \dt \dX \\
  \leqslant \max\{1,T\}^2 \vnorma{\rho}^2_{{\mathcal X}}  \int_\Theta \vnorma{g}^{2}_{\Leb^1\Iopen} \vnorma{\abs{\pdt \Phi(\X)}}^{2}_{\Leb^2\Iopen} \dX 
\end{multline*}
and thus 
\begin{equation}\label{eq:useful_estimate2}
\vnorma{\rho\partial^{\mu(t)}_t \Phi}_{\lpkIX{2}{\lp{2}}}\leqslant \max\{1,T\} T^{1-\mu}   \vnorma{\rho}_{{\mathcal X}}  \vnorma{\pdt \Phi}_{\lpkIX{2}{\lp{2}}}. 
\end{equation}
Keeping this result in mind, the variational formulation of  (\ref{eq:problem}-\ref{eq:ics_bcs}) can be formulated as:
\medskip
\begin{center}
Let AS-(1--4) be fulfilled. Find $\Phi \in \Cont\left([0,T],\lp{2}\right) \cap  \lpkIX{2}{\hko{1}}$ with $\pdt \Phi \in \lpkIX{2}{\lp{2}}$ and $\pdtt \Phi \in \lpkIX{2}{{\hko{1}}^*}$\\
such that for all $\chi \in \hko{1}$ and for a.a. $t \in (0,T)$ it holds that
 \begin{multline}\label{eq:weak_form_cont}
\langle \pdtt \Phi(t), \chi \rangle + \scal{\rho(t) \partial^{\mu(t)}_t \Phi(t)}{\chi} +  \scal{\beta(t)\nabla \Phi(t)}{\nabla \chi}\\
 = \scal{f\left(\Phi(t)\right)}{\chi} + \scal{Q(t)}{\chi}. 
 \end{multline}
\end{center}

%%%%%%%%%%%%%%%%%
\begin{remark}[Regularity on the solution: special case $\mu(t) = \mu$]
We employ the Fourier method to derive the formal solution to the following problem ($L>0$):
\begin{equation}\label{eq:problem_1D_forward}
 \left\{ \begin{array}{rlr}
  \Phi_{tt} (x,t) +  \partial^{\mu}_t \Phi(x,t) -  \Phi_{xx}(x,t) & =  0& \quad (x,t) \in (0,L)\times (0,T],\\
\Phi(0,t)=\Phi(L,t)& =0 & \quad t \in (0,T], \\
 \Phi(x,0) &  = \tilde \Phi_0(x) & \quad  x \in (0,L), \\
  \pdt \Phi(x,0) &  = \tilde\Psi_0(x) & \quad  x \in (0,L).
 \end{array} \right.
\end{equation}
Prescribing a particular solution of the form $\Phi(x,t) = X(x)T(t)$ leads to the following fractional differential equation 
\begin{equation}\label{eq:frac_diff}
 T^{\prime\prime}(t) + \left(\partial_t^{\mu} T  \right)(t) + \kappa T(t)  = 0, \quad t\in \Iopen,
\end{equation}
and the eigenvalue problem 
\begin{equation}\label{eq:eigenvalue_problem}
\begin{cases}
\left({\mathcal A}X\right)(x) =  -X^{\prime\prime}(x) = \kappa X(x) & x \in (0,L) \\
    X(0) = X(L) = 0, &  
    \end{cases}
\end{equation}
where $\kappa$ represents the separation constant. The solutions to \eqref{eq:eigenvalue_problem} are given by $\{\kappa_i,X_i\}$ for any $i\in\NN$, where $\kappa_i = \left(\frac{i\pi}{L}\right)^2$ and $X_i(x)= \sqrt{\frac{2}{L}}\sin\left(\frac{i\pi}{L}x\right)$. Now, we solve problem \eqref{eq:frac_diff} for $\kappa = \kappa_i$ by the Laplace transform method. We obtain that 
\[
{\mathcal L}\left[T_i(t)\right](z) = \frac{z+z^{\mu-1}}{z^2+z^\mu+\kappa_i} T_i(0) + \frac{1}{z^2+z^\mu+\kappa_i} T_i^\prime(0), \quad i \in \NN. 
\]
Hence, equation~\eqref{eq:frac_diff} for $\kappa = \kappa_i$ has two fundamental solutions given by (see e.g. \cite[Lemma~5]{Kazem2013} or \cite[Example~5.19]{Kilbas2006})
\begin{multline*}
   T^1_i(t) =  \sum_{k_1=0}^\infty \sum_{k_2=0}^\infty  \frac{(-\kappa_i)^{k_1} (-1)^{k_2} \binom{k_1+k_2}{k_2}}{\Gamma(k_2 (2-\mu) + 2k_1+1 )} t^{k_2 (2-\mu) + 2k_1} \\
   + t^{2-\mu} \sum_{k_1=0}^\infty \sum_{k_2=0}^\infty  \frac{(-\kappa_i)^{k_1} (-1)^{k_2} \binom{k_1+k_2}{k_2}}{\Gamma(k_2 (2-\mu) + 2k_1+3-\mu )} t^{k_2 (2-\mu) + 2k_1}
\end{multline*}
and
\[
    T^2_i(t) = t \sum_{k_1=0}^\infty \sum_{k_2=0}^\infty  \frac{(-\kappa_i)^{k_1} (-1)^{k_2} \binom{k_1+k_2}{k_2}}{\Gamma(k_2 (2-\mu) + 2k_1+2 )} t^{k_2 (2-\mu) + 2k_1}. 
\]
The multinomial Mittag-Leffler function is defined by \cite{Hadid1996}
\begin{equation} \label{eq:MML}
E_{(\alpha_1, \dots, \alpha_m),\mu}(z_1, \dots, z_m) = \sum_{k=0}^\infty \sum_{\substack{ k_1+\dots + k_m = k\\ k_j\geq 0}} \binom{k}{k_1,\dots, k_m} \frac{\prod_{j=1}^m z_j^{k_j}}{ \Gamma(\mu + \sum_{j=1}^m \alpha_j k_j)},
\end{equation} 
where $\binom{k}{k_1,\dots,k_m}$ is the multinomial coefficient. Hence, $T^1_i$ and $T^2_i$ can be rewritten in the following convenient forms 
\begin{align*}
  T^1_i(t) &=    E_{(2,2-\mu),1} \left(-\kappa_i t^2, -t^{2-\mu}\right) + t^{2-\mu} E_{(2,2-\mu),3-\mu} \left(-\kappa_i t^2, -t^{2-\mu}\right), \\
  T^2_i(t) &= t E_{(2,2-\mu),2} \left(-\kappa_i t^2, -t^{2-\mu}\right). 
\end{align*}
Using \cite[Lemma~3.2]{MaesVanBockstal}, we obtain that 
\[
 T^1_i(t) = 1 - \kappa_i t^2  E_{(2,2-\mu),3} \left(-\kappa_i t^2, -t^{2-\mu}\right).
\]
The formal solution to \eqref{eq:problem_1D_forward} is given by 
\[\Phi(x,t)= \sum_{j=1}^\infty  X_j(x) \left[\scal{\tilde{u}_0}{X_j} T^1_j(t) + \scal{\tilde\Psi_0}{X_j} T^2_j(t)\right].\]
Employing \cite[Lemma~3.3]{MaesVanBockstal}, there exist positive constants $M_1$ and $M_2$ such that
\[
\max_{t\in \I}\abs{T^i_j(t)} \leqslant M_i, \quad \forall j\in\NN, \quad i=1,2. 
\]
Hence, we have for all $t\in\I$ that
\begin{align*}
\vnorma{\Phi(t)}^2 &=  \sum_{j=1}^{\infty} \abs{(\tilde{u}_0,X_j)T_j^1(t) + (\tilde\Psi_0,X_j) T_j^2(t)}^2    \\
& \leqslant 2\max\{M_1^2,M_2^2\} \left(\vnorma{\tilde{u}_0}^2 + \vnorma{\tilde\Psi_0}^2\right).
\end{align*}
Utilising \cite[Lemma~3.1]{MaesVanBockstal} and \cite[Lemma~3.2]{MaesVanBockstal}, we further deduce that 
\begin{align*}
    \frac{\didiff}{\dt} T^1_i(t) &=  - \kappa_i t E_{(2,2-\mu),2} \left(-\kappa_i t^2, -t^{2-\mu}\right), \\
     \frac{\didiff^2}{\dt^2} T^1_i(t) &= - \kappa_i  E_{(2,2-\mu),1} \left(-\kappa_i t^2, -t^{2-\mu}\right), \\
      \frac{\didiff}{\dt} T^2_i(t) &=  E_{(2,2-\mu),1} \left(-\kappa_i t^2, -t^{2-\mu}\right), \\
       \frac{\didiff^2}{\dt^2} T^2_i(t) &= t^{-1} E_{(2,2-\mu),0} \left(-\kappa_i t^2, -t^{2-\mu}\right) \\
       &= -\kappa_i t E_{(2,2-\mu),2} \left(-\kappa_i t^2, -t^{2-\mu}\right) - t^{1-\mu} E_{(2,2-\mu),2-\mu} \left(-\kappa_i t^2, -t^{2-\mu}\right).
\end{align*}
Hence, applying \cite[Lemma~3.3]{MaesVanBockstal}, we have the existence of positive constants $\tilde{M}_i$ and $\bar{M}_i$ such that 
\[
\abs{\frac{\didiff}{\dt} T^1_i(t)} \leqslant \kappa_i^\half \tilde{M}_1 \quad \text{ and }\quad  \abs{\frac{\didiff}{\dt}  T^2_i(t)} \leqslant  \tilde{M}_2, \quad \forall i \in \NN,\; t\geqslant 0
\]
and
\[
\abs{\frac{\didiff^2}{\dt^2} T^1_i(t)} \leqslant \kappa_i \bar{M}_1 \quad \text{ and }\quad  \abs{\frac{\didiff^2}{\dt^2}  T^2_i(t)} \leqslant (1+\kappa_i^{\half}) \bar{M}_2, \quad \forall i \in \NN,\; t\geqslant 0. 
\]
Therefore, we obtain for all $t\in\I$ that 
\begin{align*}
\vnorma{\pdt \Phi(t)}^2 &=  \sum_{j=1}^{\infty} \abs{(\tilde{\Phi}_0,X_j)\left(T^1_j\right)^\prime(t) + (\tilde\Psi_0,X_j) \left(T^2_j\right)^\prime(t)}^2    \\
& \leqslant 2 \tilde{M}_1^2 \sum_{j=1}^{\infty}  \kappa_j \abs{(\tilde{\Phi}_0, X_j)}^2 + 2 \tilde{M}_2^2 \sum_{j=1}^\infty\abs{(\tilde\Psi_0,X_j)}^2 \\
&\leqslant 2\max\{\tilde{M}_1^2,\tilde{M}_2^2\} \left(   \vnorma{\tilde{\Phi}_0}_{\Hi^1_0(0,L)}^2+ \vnorma{\tilde\Psi_0}^2 \right)
\end{align*}
and 
\begin{align*}
\vnorma{\pdtt \Phi(t)}^2 &=  \sum_{j=1}^{\infty} \abs{(\tilde{\Phi}_0,X_j)\left(T^1_j\right)^{\prime\prime}(t) + (\tilde\Psi_0,X_j) \left(T^2_j\right)^{\prime\prime}(t)}^2    \\
& \leqslant 2 \bar{M}_1^2 \sum_{j=1}^{\infty}  \kappa_j^2 \abs{(\tilde{\Phi}_0, X_j)}^2 + 2 \bar{M}_2^2 \sum_{j=1}^\infty \kappa_j\abs{(\tilde\Psi_0,X_j)}^2 \\
&\leqslant 2\max\{\bar{M}_1^2,\bar{M}_2^2\} \left(   \vnorma{\tilde{\Phi}_0}_{\Hi^2(0,L) \cap \Hi^1_0(0,L)}^2+ \vnorma{\tilde\Psi_0}_{\Hi^1_0(0,L)}^2 \right).  
\end{align*}
Since $X_j^\prime(x) = \kappa_j^{\half} \sqrt{\frac{2}{L}}\cos\left(\frac{j\pi}{L}x\right) $ for any $j\in \NN$, we have for all $t\in\I$ that 
\begin{align*}
\vnorma{\partial^2_{tx} \Phi(t)}^2 &= 
\sum_{j=1}^{\infty} \kappa_j \abs{(\tilde{\Phi}_0,X_j)\left(T^1_j\right)^\prime(t) + (\tilde\Psi_0,X_j) \left(T^2_j\right)^\prime(t)}^2  \\
& \leqslant 2 \tilde{M}_1^2 \sum_{j=1}^{\infty}  \kappa_j^2 \abs{(\tilde{\Phi}_0, X_j)}^2 + 2 \tilde{M}_2^2 \sum_{j=1}^\infty \kappa_j\abs{(\tilde\Psi_0,X_j)}^2 \\
&\leqslant 2\max\{\tilde{M}_1^2,\tilde{M}_2^2\} \left(   \vnorma{\tilde{\Phi}_0}_{\Hi^2(0,L) \cap \Hi^1_0(0,L)}^2+ \vnorma{\tilde\Psi_0}_{\Hi^1_0(0,L)}^2 \right), 
\end{align*}
i.e. $\pdt \Phi \in \lpkIX{\infty}{\hko{1}}$ if $\tilde{\Phi}_0\in \Hi^2(0,L) \cap \Hi^1_0(0,L)$ and $\tilde\Psi_0\in \Hi^1_0(0,L)$. 
\end{remark}

%%%%%%%%%%
\subsection{Uniqueness of a solution}
%%%%%%%%%%

In the following theorem, we discuss the uniqueness of a solution under the extra assumption that $\nabla \pdt \Phi \in \lpkIX{2}{\Lp{2}}$. Consequently, as\\ $\pdtt \Phi \in \lpkIX{2}{{\hko{1}}^*}$, we have as well that $\pdt \Phi \in \cIX{\lp{2}}$, see e.g. \cite[Lemma~7.3]{Roubicek2005}.

\begin{theorem}[Uniqueness of a weak solution]\label{thm:uniqueness1}
There exists at most one solution to problem \eqref{eq:weak_form_cont} fulfilling $\Phi \in \Cont\left([0,T],\lp{2}\right) \cap  \lpkIX{2}{\hko{1}}$ with $\pdt \Phi \in \lpkIX{2}{\hko{1}} \cap \cIX{\lp{2}}$ and $\pdtt \Phi \in \lpkIX{2}{{\hko{1}}^*}$. 
\end{theorem}

\begin{proof}
Let $\Phi_1$ and $\Phi_2$ be two solutions to problem \eqref{eq:weak_form_cont}. Then, the difference $\Phi=\Phi_1-\Phi_2$ satisfies $\Phi(\cdot, 0)=\pdt \Phi(\cdot,0)=0$ in $\Theta$. Now, we subtract \eqref{eq:weak_form_cont} with $\Phi=\Phi_2$ from \eqref{eq:weak_form_cont} with $\Phi=\Phi_1$, take $\chi = \pdt \Phi(t)$ and integrate over $(0,\theta)\subset \Iopen$ to obtain that 
\begin{multline*}
\half \vnorma{\pdt \Phi(\theta)}^2 + \int_0^\theta  \scal{\beta(t)\nabla \Phi(t)}{\nabla \pdt \Phi(t)} \dt \\ = - \int_0^\theta \scal{\rho(t) \partial^{\mu(t)}_t \Phi(t)}{\pdt \Phi(t)}\dt + \int_0^\theta \scal{f\left(\Phi_1(t)\right)-f\left(\Phi_2(t)\right)}{\pdt \Phi(t)} \dt. 
\end{multline*}
Using estimate \eqref{eq:useful_estimate2}, we obtain that 
\begin{multline*}
    \abs{ \int_0^\theta \scal{\rho(t) \partial^{\mu(t)}_t \Phi(t)}{\pdt \Phi(t)}\dt } \\
 \leqslant \left(\half + \half \max\{1,T\}^2 T^{2-2\mu}  \vnorma{\rho}^2_{{\mathcal X}} \right) \int_0^\theta  \vnorma{\pdt \Phi(s)}^2 \ds.
\end{multline*} 
Applying $\beta = \beta^\half \beta^\half$, AS-5 and AS-6, we get that 
\begin{align*}
   & \int_0^\theta  \scal{\beta(t)\nabla \Phi(t)}{\nabla \pdt \Phi(t)} \dt \\
   &=  \int_0^\theta  \scal{\beta^{1/2}(t)\nabla \Phi(t)}{\pdt \left(\beta^{1/2}(t) \nabla \Phi(t)\right) - \pdt \beta^{1/2}(t) \nabla \Phi(t) } \dt \\
   &\geqslant \frac{\tilde{\beta}_0}{2} \vnorma{\nabla \Phi(\theta)}^2 - \frac{\vnorma{\pdt \beta}_{{\mathcal X}}}{2} \int_0^\theta \vnorma{\Phi(t)}_{\hk{1}}^2 \dt. 
\end{align*}
By the Young inequality, AS-4 and $\Phi(t) = \int_0^t \pdt \Phi(s)\ds$, we have that 
\[
\abs{ \int_0^\theta \scal{f\left(\Phi_1(t)\right)-f\left(\Phi_2(t)\right)}{\pdt \Phi(t)} \dt } 
\leqslant \frac{1+L_f^2 T^2}{2}  \int_0^\theta \vnorma{\pdt \Phi(s)}^2 \ds. 
\]
Collecting these estimates, we obtain that
\begin{multline*}
   \half \vnorma{\pdt \Phi(\theta)}^2  +  \frac{\tilde{\beta}_0}{2} \vnorma{\nabla \Phi(\theta)}^2
    \leqslant   \frac{\vnorma{\pdt \beta}_{{\mathcal X}}}{2} \int_0^\theta \vnorma{\Phi(t)}_{\hk{1}}^2 \dt\\ + 
    \left(1+ \frac{L_f^2 T^2}{2} + \half \max\{1,T\}^2 T^{2-2\mu}  \vnorma{\rho}^2_{{\mathcal X}}\right) \int_0^\theta \vnorma{\pdt \Phi(s)}^2 \ds. 
\end{multline*}
Therefore, an application of the Gr\"onwall lemma gives that $\vnorma{\pdt \Phi(\theta)}^2=0$ and thus $\Phi=0$ a.e. in $\Theta \times \Iopen.$
\end{proof}

%%%%%%%%%%%%%%%%%%%%%%%%%%%%%%%%%%%%%
%%%%%%%%%%%%%%%%%%%%%%%%%%%%%%%%%%%%%
\section{Time discretization}
\label{sec:time_discretization}
%%%%%%%%%%%%%%%%%%%%%%%%%%%%%%%%%%%%%
%%%%%%%%%%%%%%%%%%%%%%%%%%%%%%%%%%%%%

We consider an equidistant partition of the time interval $\I$ into $n$ intervals with length $\tau = \frac{T}{n}<1$. The solution at time $t_i$ is given by $\Phi_i$, whilst the first derivative $\pdt \Phi(t_i)$ can be  approximated by the  Euler backward difference: 
\[\pdt \Phi(t_i) \approx \delta \Phi_i := 
\frac{\Phi_i-\Phi_{i-1}}{\tau}.\]
Additionally, the V-OF Caputo operator $\frac {\partial ^{\mu(t)} \Phi(t)} {\partial t ^{\mu(t)}}$ at time $t_i$ is approximated by 
\begin{multline}
 \left. \frac {\partial ^{\mu(t)} \Phi(t)} {\partial t ^{\mu(t)}} \right\rvert_{t = t_k} =  \sum_{q=1}^k \int_{t_{q-1}}^{t_q} \frac{(t_k-r)^{-\mu_k}}{\Gamma\left(1-\mu_k\right)}\pdt \Phi(r)\dr \\
\approx D_{\tau}^{\mu_k}\Phi_k  := \sum_{q=1}^k a_q^k\delta \Phi_q \tau = \sum_{q=0}^k b^k_q \Phi_q,\
\end{multline}
where 
\begin{multline*}
    a_q^k=\frac{(t_k-t_{q-1})^{-\mu_k}}{\Gamma\left(1-\mu_k\right)},\\ b _0^{k} = -a^{k} _1, \quad b^{k} _k =  a^{k} _{k }, \quad b^{k} _{q} = a^{k} _{q} - a^{k} _{q + 1},\, \text{for } \,q = 1 , \ldots , k - 1.
\end{multline*}
Using these approximations, the problem \eqref{eq:weak_form_cont} is approximated at time $t_i$ as follows: 
\begin{center}
Given $\Phi_0:=\tilde{\Phi}_0$ and $\delta \Phi_0 := \tilde\Psi_0$. Find $\Phi_i \in \hko{1}$ such that 
\begin{multline}\label{eq:var_for_discrete}
 \scal{\delta^2 \Phi_i}{\chi} + \scal{\rho_i D_{\tau}^{\mu_i}\Phi_i}{\chi}  +  \scal{\beta_i \nabla \Phi_i}{\nabla \chi} \\
 = \scal{f\left(\Phi_{i-1}\right)}{\chi} + \scal{Q_i}{\chi}, \quad \forall \chi \in \hko{1}.
\end{multline}
\end{center}
It is equivalent to  solving
\begin{align*}\label{eq:var_for_discrete2}
a_i(\Phi_i,\chi)=\langle l_i,\chi\rangle, \quad \forall \chi \in \hko{1},
\end{align*}
where
$$
a_i(\Phi_i,\chi) := \tau^{-2} \scal{\Phi_i}{\chi} + \frac{\tau^{-\mu_i}}{\Gamma(1-\mu_i)} \scal{\rho_i \Phi_i}{\chi}  +  \scal{\beta_i \nabla \Phi_i}{\nabla \chi}
$$
and 
\begin{multline*}
\langle l_i,\chi\rangle := \scal{f\left(\Phi_{i-1}\right)}{\chi} + \scal{Q_i}{\chi} + \tau^{-2} \scal{\Phi_{i-1}}{\chi} + \tau^{-1} \scal{\delta \Phi_{i-1}}{\chi} \\
+ \frac{\tau^{-\mu_i}}{\Gamma(1-\mu_i)} \scal{\rho_i \Phi_{i-1}}{\chi} 
- \sum_{l=1}^{i-1}  \frac{(t_i-t_{l-1})^{-\mu_i}}{\Gamma\left(1-\mu_i\right)} \scal{\rho_i\delta \Phi_l}{\chi} \tau. 
\end{multline*}

In the next lemma, we establish the existence of a unique solution to \eqref{eq:var_for_discrete}.

%%%%%%
\begin{lemma}
\label{lemnn}
Let the assumptions AS-(1--8) be fulfilled. 
Then, for any $i=1,2,\ldots,n$, there exists a unique $\Phi_i\in \hko{1}$ solving \eqref{eq:var_for_discrete}.
\end{lemma}
%%%%%%
%%%%%%
\begin{proof}
The conditions of the Lax-Milgram lemma are satisfied assuming that $\tilde{\Phi}_0,\tilde\Psi_0\in\lp{2}$. The ellipticity of the bounded bilinear form $a_i$ follows from Friedrichs' inequality and AS-7. The linear form $l_i$ is bounded if $\delta \Phi_{i-1} \in \lp{2}$ and $\Phi_l \in \lp{2}$ for $l=0,\ldots,i-1$ since for all $\chi \in \hko{1}$ it holds that
\begin{align*}
\abs{\langle l_i,\chi\rangle}&\leqslant C  \vnorma{\chi}  + C(1+\tau^{-2}+\tau^{-\mu_i})\vnorma{\Phi_{i-1}}\vnorma{\chi} \\
& \qquad + \tau^{-1} \vnorma{\delta \Phi_{i-1}}\vnorma{\chi} + C \tau^{-\mu_i} \vnorma{\chi}\sum_{l=1}^{i-1} \vnorma{ \Phi_l- \Phi_{l-1}}  \\
& \leqslant C(\tau^{-2}) \vnorma{\chi}_{\hk{1}}.
\end{align*}
\end{proof}
%%%%%%%

Next, we derive some stability estimates associated with the time-discrete solutions. 

%%%%%%%%%%%%%%
\begin{lemma}\label{lem:est1} 
Let AS-(1--8) be satisfied. Then, there exist positive constants $C>0$ and $\tau_0>0$ such that $\forall j=1,2,\ldots,n$ and $\tau < \tau_0$ it holds that
$$
   \vnorma{\delta \Phi_j}^2 
      +  \sum_{i=1}^j \vnorma{\delta \Phi_i -\delta \Phi_{i-1}}^2 
      +  \vnorma{\nabla \Phi_j}^2  +  \sum_{i=1}^j \vnorma{\nabla \Phi_i - \nabla \Phi_{i-1}}^2 \\
      \leqslant  C.
$$
\end{lemma}
%%%%%%%%%%%%%%
%%%%%%%%%%%%%%
\begin{proof}
We start with putting $\chi=\delta \Phi_i\tau$ in \eqref{eq:var_for_discrete} and summing up the outcome for $i=1,\ldots,j$ with $1\leqslant j \leqslant n$, i.e. 
 \begin{multline*}%\label{lem:est1:eq1}
\sum_{i=1}^j \scal{\delta^2 \Phi_i}{\delta \Phi_i} \tau +  \sum_{i=1}^j \scal{ \rho_i \sum_{l=1}^i \frac{(t_i-t_{l-1})^{-\mu_i}}{\Gamma\left(1-\mu_i\right)}\delta \Phi_l \tau}{\delta \Phi_i} \tau +   \sum_{i=1}^j \scal{\beta _i\nabla \Phi_i}{\nabla \delta \Phi_i}\tau \\
 = \sum_{i=1}^j \scal{f\left(\Phi_{i-1}\right) }{\delta \Phi_i}\tau +  \sum_{i=1}^j \scal{Q_i}{\delta \Phi_i} \tau. 
 \end{multline*}
 Abel's summation rule implies that
 \[
\sum_{i=1}^j \scal{\delta^2 \Phi_i}{\delta \Phi_i} \tau = \half \vnorma{\delta \Phi_j}^2 - \half \vnorma{\tilde\Psi_0}^2 + \half \sum_{i=1}^j \vnorma{\delta \Phi_i -\delta \Phi_{i-1}}^2. 
\]
Moreover, from
\begin{multline*}
  \sum_{i=1}^j  \scal{\beta_i \nabla \Phi_i}{\nabla \delta \Phi_i} \tau 
    = \half \scal{\beta_j \nabla \Phi_j}{\nabla \Phi_j}
    - \half \scal{\beta_0 \nabla \tilde{\Phi}_0}{\nabla  \tilde{\Phi}_0}\\
    + \half \sum_{i=1}^j \scal{\beta_i (\nabla \Phi_i -  \nabla \Phi_{i-1})}{\nabla \Phi_i - \nabla \Phi_{i-1}}
   - \half \sum_{i=1}^j \scal{\delta\beta_i \nabla \Phi_{i-1}}{\nabla \Phi_{i-1}} \tau,
\end{multline*}
we have that
\begin{multline*}
  \sum_{i=1}^j  \scal{\beta_i \nabla \Phi_i}{\nabla \delta \Phi_i} \tau 
  \geqslant \frac{\tilde{\beta}_0}{2}  \vnorma{\nabla \Phi_j}^2 - \frac{\vnorma{ \beta}_{{\mathcal X}}}{2} \vnorma{\nabla \tilde{\Phi}_0}^2 \\
							   + \frac{\tilde{\beta}_0}{2} \sum_{i=1}^j \vnorma{\nabla \Phi_i - \nabla \Phi_{i-1}}^2 -\frac{\vnorma{\pdt \beta}_{{\mathcal X}}}{2}  \sum_{i=1}^{j-1} \vnorma{\nabla \Phi_{i}}^2 \tau.
\end{multline*}
Next, Young's inequality implies that 
\[
\left\lvert\sum_{i=1}^j \scal{Q_i}{\delta \Phi_i} \tau \right\rvert  \leqslant C + \half \sum_{i=1}^j \vnorma{\delta \Phi_i}^2 \tau.
\]
Using AS-4 and $\Phi_i = \tilde{\Phi}_0 + \sum_{l=1}^i \delta \Phi_l \tau$,  we see that
\begin{align*}
   \left\lvert  \sum_{i=1}^j \scal{f\left(\Phi_{i-1}\right) }{\delta \Phi_i}\tau \right\rvert 
   & \leqslant \half  \sum_{i=1}^j \vnorma{\delta \Phi_i}^2 \tau +  C\sum_{i=1}^j \left(1 + \vnorma{\Phi_{i-1}}^2  \right) \tau \\
   & \leqslant  C  + C  \sum_{i=1}^j \vnorma{\delta \Phi_i}^2 \tau. 
\end{align*}
Moreover, using \eqref{eq:useful_estimate}, we have that 
\begin{multline*}
  \left\lvert  \sum_{i=1}^j \scal{ \rho_i \sum_{l=1}^i \frac{(t_i-t_{l-1})^{-\mu_i}}{\Gamma\left(1-\mu_i\right)}\delta \Phi_l \tau}{\delta \Phi_i} \tau \right\rvert 
 \leqslant  \half  \sum_{i=1}^j \vnorma{\delta \Phi_i}^2 \tau \\+  \half \max\{1,T\}^2  \vnorma{\rho}^2_{{\mathcal X}} \int_\Theta  \sum_{i=1}^j \tau \left(\sum_{l=1}^i t_{i-l+1}^{-\bar{\mu}} \abs{\delta \Phi_l(\X)} \tau\right)^2 \dX.
\end{multline*}
Now, we define $b_l=g_{l+1}\tau$ 
%for $l\in \{0,\ldots,j-1\}$ 
%with $b_l=0$ for $l\not\in \{0,\ldots,j-1\}$, 
and $c_l(\X)=\abs{\delta \Phi_l(\X)}\sqrt{\tau}$, then 
\[
\int_\Theta  \sum_{i=1}^j \tau \left(\sum_{l=1}^i t_{i-l+1}^{-\bar{\mu}} \abs{\delta \Phi_l(\X)} \tau\right)^2 \dX = \int_\Theta \sum_{i=1}^j \left(b \ast_d (c(\X))\right)_i^2 \dX,
\]
where $\left(b\ast_d c\right)_i = \sum_{l=1}^i b_{i-l} c_l$. By the discrete Young's inequality for convolutions
\begin{equation}\label{eq:young_conv_discrete}
    \left(\sum_{i=1}^j \abs{(b\ast_d c)_i}^r \right)^{\frac{1}{r}}  \leqslant \left(\sum_{i=0}^{j-1} \abs{b_i}^p \right)^{\frac{1}{p}} \left(\sum_{i=1}^j \abs{c_i}^q \right)^{\frac{1}{q}}, \quad j\in \NN, 
\end{equation}
where $\frac{1}{p} + \frac{1}{q} = \frac{1}{r}+1$ with $1 \leqslant p,q \leqslant r \leqslant \infty$,
we obtain that 
\[
\abs{\int_\Theta  \sum_{i=1}^j \tau \left(\sum_{l=1}^i t_{i-l+1}^{-\bar{\mu}} \abs{\delta \Phi_l(\X)} \tau\right)^2 \dX}
\leqslant \int_\Theta \left(\sum_{i=0}^{j-1} \abs{b_i}\right)^2\left(\sum_{i=1}^j c_i(\X)^2\right) \dX. 
\]
As 
\[
\sum_{i=0}^{j-1} \abs{ b_i } = 
\sum_{i=1}^j t_{i}^{-\bar{\mu}} \tau
\leqslant \sum_{i=1}^j \int_{t_{i-1}}^{t_i} s^{-\bar{\mu}} \ds 
= \int_0^{t_j} s^{-\bar{\mu}} \ds \leqslant T^{1-\bar{\mu}},  
\]
we finally have that 
\[
\abs{\int_\Theta  \sum_{i=1}^j \tau \left(\sum_{l=1}^i t_{i-l+1}^{-\bar{\mu}} \abs{\delta \Phi_l(\X)} \tau\right)^2 \dX}
\leqslant T^{2-2\bar{\mu}}  \sum_{i=1}^j \vnorma{\delta \Phi_i}^2 \tau,
\]
and thus 
\begin{multline*}
  \left\lvert  \sum_{i=1}^j \scal{ \rho_i \sum_{l=1}^i \frac{(t_i-t_{l-1})^{-\mu_i}}{\Gamma\left(1-\mu_i\right)}\delta \Phi_l \tau}{\delta \Phi_i} \tau \right\rvert \\
 \leqslant \half\left(1 + \max\{1,T\}^2  \vnorma{\rho}^2_{{\mathcal X}}  T^{2-2\bar{\mu}} \right) \sum_{i=1}^j \vnorma{\delta \Phi_i}^2 \tau. 
 \end{multline*}
Now, we collect all the estimates we have above calculated to obtain that
\begin{multline*}
   \vnorma{\delta \Phi_j}^2 
      +  \sum_{i=1}^j \vnorma{\delta \Phi_i -\delta \Phi_{i-1}}^2 
      +  \vnorma{\nabla \Phi_j}^2  +  \sum_{i=1}^j \vnorma{\nabla \Phi_i - \nabla \Phi_{i-1}}^2 \\
      \leqslant  C  + C \sum_{i=1}^{j-1} \vnorma{\nabla \Phi_{i}}^2 \tau + C  \sum_{i=1}^j \vnorma{\delta \Phi_i}^2 \tau. 
\end{multline*}
An application of the classical discrete Gr\"onwall lemma concludes the proof. 
 \end{proof}

Considering the proof of the previous theorem, we have that 
\[
\vnorma{\Phi_i}^2 \leqslant 2 \vnorma{\tilde{\Phi}_0}^2 + 2 T \sum_{i=1}^j \vnorma{\delta \Phi_i}^2 \tau
\]
and
\[
\sum_{i=1}^j \vnorma{ D_{\tau}^{\mu_i}\Phi_i}^2 \tau \leqslant  \max\{1,T\}^2   T^{2-2\bar{\mu}} \sum_{i=1}^j \vnorma{\delta \Phi_i}^2 \tau.
\]
Hence, using the result of Lemma~\ref{lem:est1}, we immediately obtain the following result. 

%%%%%%%%
\begin{lemma}\label{lem:est2}
Let AS-(1--8) be satisfied. Then, there exist positive constants $C>0$ and $\tau_0>0$ such that $\forall j=1,2,\ldots,n$ and $\tau < \tau_0$ it holds that
\[
\vnorma{\Phi_i}^2 + \sum_{i=1}^j \vnorma{ D_{\tau}^{\mu_i}\Phi_i}^2 \tau \leqslant C.
\]
\end{lemma}
%%%%%%%%

The previous lemmas lead to the following consequence. 
%%%%%%%%
\begin{corollary}\label{cor:est1}
Let AS-(1--8) be fulfilled. Then, there exist positive constants $C>0$ and $\tau_0>0$ such that $\forall j=1,2,\ldots,n$ and $\tau < \tau_0$ it holds that
\[
\sum_{i=1}^j \vnorma{ \delta^2 \Phi_i}^2_{{\hko{1}}^*} \tau \leqslant C.
\]
\end{corollary}
%%%%%%%%

\begin{proof}
The result follows from Lemma~\ref{lem:est1} and \ref{lem:est2} since
 \begin{align*}
  \vnorma{ \delta^2 \Phi_i}_{{\hko{1}}^*} &= \sup_{\vnorma{\chi}_{\hko{1}}=1} 
 \abs{\langle \delta^2 \Phi_i, \chi \rangle_{{\hko{1}}^* \times \hko{1}}}\\
 & = \sup_{\vnorma{\chi}_{\hko{1}}=1} \abs{   \scal{f\left(\Phi_{i-1}\right)}{\chi} + \scal{Q_i}{\chi} - \scal{\rho_i D_{\tau}^{\mu_i}\Phi_i}{\chi}  -   \scal{\beta_i \nabla \Phi_i}{\nabla \chi} } \\
 &\leqslant C \left(1 + \vnorma{\Phi_{i-1}}  + \vnorma{Q_{i}} +\vnorma{D_{\tau}^{\mu_i}\Phi_i} + \vnorma{\nabla \Phi_i}\right).
\end{align*}
\end{proof}

%%%%%%%%%%%%%%%%%%%%
\section{Existence of a solution}
\label{sec:existence}
%%%%%%%%%%%%%%%%%%%%%%%%

In order to be able to show the existence of a solution, we first prolong the solutions on the single time steps on the whole time frame with aid of the Rothe functions
\[\overline U_n:[0,T] \to\lp{2}:t\mapsto\begin{cases} \tilde{\Phi}_0 &  t \in [-\tau,0],\\
 \Phi_i &     t\in (t_{i-1},t_i],\quad 1\leqslant i\leqslant n;\end{cases}\]
 \[\overline 
 V_n :[0,T] \to\lp{2}:t\mapsto\begin{cases} \tilde\Psi_0 & t = 0,\\
 \delta \Phi_i &     t\in (t_{i-1},t_i],\quad 1\leqslant i\leqslant n;\end{cases}\]
\[U_n:[0,T]\to\lp{2}:t\mapsto\begin{cases} \tilde{\Phi}_0 & t \in [-\tau,0],\\
\Phi_{i-1} + (t-t_{i-1})\delta \Phi_i &
     t\in (t_{i-1},t_i],\quad 1\leqslant i\leqslant n; \end{cases}\]
and 
\[V_n:[0,T]\to\lp{2}:t\mapsto\begin{cases} \tilde\Psi_0 & t = 0,\\
\delta \Phi_{i-1} + (t-t_{i-1})\delta^2 \Phi_i &
     t\in (t_{i-1},t_i],\quad 1\leqslant i\leqslant n. \end{cases}\]
Similarly, we define the functions $\overline{B}_n$, $\overline{K}_n$, $\overline{P}_n$  and $\overline{Q}_n$ in connection with the given functions $\mu$, $\beta$, $\rho$ and $Q,$ respectively. Moreover, we define $\overline{d}_n(t_k,t)= t_k-t_i$ for $t\in(t_{i-1},t_i]$ and $k>i$.
Using these functions and the notation $\lceil t\rceil_{\tau} = t_i$ for $t\in(t_{i-1},t_i]$, we reformulate \eqref{eq:var_for_discrete} for all $\chi \in \hko{1}$ on $(0,T]$ as 
\begin{multline}\label{eq:varfor_whole_time_frame}
 \scal{\pdt V_n(t)}{ \chi} +  \scal{\overline{P}_n(t) \int_0^{\lceil t\rceil_{\tau}} \frac{ \overline{d}_n(\lceil t\rceil_{\tau},r-\tau)^{-\overline{B}_n(t)}}{\Gamma\left(1-\overline{B}_n(t)\right)} \partial_r U_n(r)\dr }{\chi} \\
  + \scal{\overline{K}_n(t) \nabla\overline U_n(t)}{\nabla\chi}
  = \scal{f\left( \overline{U}_n(t-\tau)\right)}{\chi} +  \scal{\overline{Q}_n(t)}{\chi}.
\end{multline}

 %%%%%%%%%%%%%%%%%%%%%%%%%%%
 \begin{theorem}[Existence]\label{thm:existence}
Let AS-(1--8) be fulfilled.  Then, a weak solution $\Phi$ exists to \eqref{eq:weak_form_cont} satisfying 
 \begin{multline*}
 \Phi \in \Lip\left([0,T], \lp{2}\right)  \cap \lpkIX{\infty}{\hko{1}}, \\
 \pdt \Phi \in \lpkIX{\infty}{\lp{2}}, \pdtt \Phi \in \lpkIX{2}{{\hko{1}}^*}. 
 \end{multline*}
 \end{theorem}
 %%%%%%%%%%%%%%%%%%%%%%%%

\begin{proof}
The starting point of the proof is rather standard. We first recall the compact embedding   $\hko{1} \imbed \imbed \lp{2} $, see e.g. \cite[Theorem 6.6-3]{Ciarlet2013}. 
Lemma~\ref{lem:est1} and \ref{lem:est2} imply for all $n\geqslant n_0 > 0$ that 
\[
\max_{t\in [0,T]} \left\{ \vnorma{\overline{U}_n(t)}_{\hk{1}}^2  + \vnorma{\pdt {U}_n(t)}^2 \right\} \leqslant C.
\]
Then, \cite[Lemma 1.3.13]{Kacur1985} leads to the existence of a function $\Phi \in \Lip\left([0,T], \lp{2}\right)  \cap \Leb^{\infty}\left((0,T), \hko{1}\right)$ with $\pdt \Phi \in \lpkIX{\infty}{\lp{2}}$,
and a subsequence $\{ U_{n_l}\}_{l\in\NN}$ of $\{U_n\}$ such that
\begin{displaymath}
\left\{
\begin{array}{ll}
U_{n_l} \to \Phi & \text{in}~~\Cont\left([0,T], \lp{2}\right) , \\[4pt]
U_{n_l}(t) \rightharpoonup \Phi(t) & \text{in}~~\hko{1},~~\forall t \in [0,T], \\[4pt]
\overline{U}_{n_l}(t) \rightharpoonup \Phi(t) & \text{in}~~\hko{1},~~\forall t \in [0,T],  \\[4pt]
\pdt {U}_{n_l} \rightharpoonup \pdt \Phi & \text{in}~~\Leb^{2}\left((0,T), \lp{2}\right).
\end{array}
\right.
\end{displaymath}
 From  Lemma~\ref{lem:est1}, we have for all $t\in (t_{i-1},t_i]$ that 
\begin{equation}\label{thm:existence:eq2}
\vnorma{\overline{U}_{n_l}(t) - U_{n_l}(t) }^2  +  \vnorma{\overline{U}_{n_l}(t-\tau_{n_l}) - \overline{U}_{n_l}(t) }^2  \leqslant 2 \tau_{n_l}^2  \vnorma{\delta \Phi_i}^2 \leqslant C \tau_{n_l}^2,
\end{equation}
i.e.  $\overline U_{n_l} \to \Phi$  and $\overline U_{n_l}(\cdot -\tau_{n_l}) \to \Phi$ in $\Leb^{2}\left((0,T), \lp{2}\right)$ as $l\to \infty$. \\
Now, we follow a similar reasoning for the corresponding sequences $\overline{V}_{n_l}$ and $V_{n_l}$. Employing the results of Lemma~\ref{lem:est1} and Corollary~\ref{cor:est1}, we have that
\[
\max_{t\in [0,T]} \vnorma{\overline{V}_{n_l}(t)}^2  + \int_0^T \vnorma{\pdt {V}_{n_l}(t)}_{{\hko{1}}^*}^2 \dt \leqslant C.
\]
The compactness argument $\lp{2} \imbed \imbed {\hko{1}}^*$ and noticing that $\pdt {U}_{n_l} = \overline{V}_{n_l} $ yield that the function $\Phi$ obtained before satisfies that (for a subsequence denoted by the same index $n_l$)
% a function $u \in \Lip\left([0,T], \lp{2}\right)  \cap \Leb^{\infty}\left((0,T), \hko{1}\right)$ with $\pdt u \in \lpkIX{\infty}{\lp{2}}$, and a subsequence $\{ V_{{n_l}_k}\}_{k\in\NN}$ of $\{V_{n_l}\}$ (denoted by the $\{V_{n_l}\}$ itself) such that
\begin{displaymath}
\left\{
\begin{array}{ll}
V_{n_l} \to \pdt \Phi & \text{in}~~\Cont\left([0,T], {\hko{1}}^*\right) , \\[4pt]
V_{n_l}(t) \rightharpoonup \pdt \Phi(t) & \text{in}~~\lp{2},~~\forall t \in [0,T], \\[4pt]
%\overline{V}_{n_l}(t) \rightharpoonup \Phi(t) & \text{in}~~\lp{2},~~\forall t \in [0,T],  \\[4pt]
\pdt {V}_{n_l} \rightharpoonup \pdtt \Phi & \text{in}~~\Leb^{2}\left((0,T), {\hko{1}}^*\right).
\end{array}
\right.
\end{displaymath}
We integrate \eqref{eq:varfor_whole_time_frame} for $n=n_l$ over $(0,\eta)\subset \Iopen$ and obtain that 
\begin{multline}\label{eq:varfor_whole_time_frame_int}
\int_0^{\eta} \scal{\pdt V_{n_l}(t)}{ \chi} \dt  + \int_0^{\eta} \scal{\overline{P}_{n_l}(t) \int_0^{\lceil t\rceil_{\tau_{n_l}}} \frac{ \overline{d}_{n_l}(\lceil t\rceil_{\tau_{n_l}},r-\tau_{n_l})^{-\overline{B}_{n_l}(t)}}{\Gamma\left(1-\overline{B}_{n_l}(t)\right)} \partial_r U_{n_l}(r)\dr }{\chi} \dt \\
  + \int_0^{\eta} \scal{\overline{K}_{n_l}(t) \nabla\overline U_{n_l}(t)}{\nabla\chi} \dt
  = 
  \int_0^{\eta} \scal{f\left( \overline{U}_{n_l}(t-{\tau_{n_l}})\right)}{\chi} \dt \\
  +  \int_0^{\eta} \scal{\overline{Q}_{n_l}(t)}{\chi} \dt.
\end{multline}
We only point out the limit transition of
\[
{\mathcal T} := \int_0^{\eta} \scal{\overline{P}_{n_l}(t) \int_0^{\lceil t\rceil_{{\tau_{n_l}}}} \frac{ \overline{d}_{n_l}(\lceil t\rceil_{\tau_{n_l}},r-\tau_{n_l})^{-\overline{B}_{n_l}(t)}}{\Gamma\left(1-\overline{B}_{n_l}(t)\right)}  \partial_r U_{n_l}(r)\dr }{\chi} \dt,  
\]
since the limit transition of the other terms is standard. We rewrite ${\mathcal T}$ as follows
\[
{\mathcal T} = \sum_{i=1}^5 {\mathcal T}_i,
\]
with 
\begin{align*}
    {\mathcal T}_1 & := \int_0^{\eta} \scal{\overline{P}_{n_l}(t) \int_t^{\lceil t\rceil_{{\tau_{n_l}}}} \frac{ \overline{d}_{n_l}(\lceil t\rceil_{\tau_{n_l}},r-\tau_{n_l})^{-\overline{B}_{n_l}(t)}}{\Gamma\left(1-\overline{B}_{n_l}(t)\right)} \partial_r U_{n_l}(r)\dr }{\chi} \dt, \\
    {\mathcal T}_2 &:= \int_0^{\eta} \scal{\overline{P}_{n_l}(t) \int_0^{t} \frac{ \overline{d}_{n_l}(\lceil t\rceil_{\tau_{n_l}},r-\tau_{n_l})^{-\overline{B}_{n_l}(t)}}{\Gamma\left(1-\overline{B}_{n_l}(t)\right)} \partial_r U_{n_l}(r)\dr }{\chi} \dt  \\
    & \qquad - \int_0^{\eta} \scal{\overline{P}_{n_l}(t) \int_0^{t} \frac{ (t-r)^{-\mu(t)}}{\Gamma\left(1-\mu(t)\right)} \partial_r U_{n_l}(r)\dr }{\chi} \dt, \\
 {\mathcal T}_3  &  := \int_0^{\eta} \scal{\overline{P}_{n_l}(t) \int_0^{t} \frac{ (t-r)^{-\mu(t)}}{\Gamma\left(1-\mu(t)\right)} \partial_r U_{n_l}(r)\dr }{\chi} \dt \\
    & \qquad - \int_0^{\eta} \scal{\rho(t)  \int_0^{t} \frac{ (t-r)^{-\mu(t)}}{\Gamma\left(1-\mu(t)\right)} \partial_r U_{n_l}(r)\dr }{\chi} \dt, \\
     {\mathcal T}_4 & := \int_0^{\eta} \scal{\rho(t)  \int_0^{t} \frac{ (t-r)^{-\mu(t)}}{\Gamma\left(1-\mu(t)\right)} \partial_r U_{n_l}(r)\dr }{\chi} \dt \\
    & \qquad - \int_0^{\eta} \scal{\rho(t) \int_0^{t} \frac{ (t-r)^{-\mu(t)}}{\Gamma\left(1-\mu(t)\right)} \pdt \Phi(r)\dr }{\chi} \dt, \\
  {\mathcal T}_5 & :=  \int_0^{\eta} \scal{\rho(t) \int_0^{t} \frac{ (t-r)^{-\mu(t)}}{\Gamma\left(1-\mu(t)\right)} \pdt \Phi(r)\dr }{\chi} \dt.
\end{align*}
We will show that 
\[
\abs{{\mathcal T} - {\mathcal T}_5} \leqslant \sum_{i=1}^4 \abs{{\mathcal T}_i } \to 0 \quad \text{ as } \quad l \to \infty. 
\]
For ${\mathcal T}_1$, we see that 
\begin{align*}
\abs{{\mathcal T}_1} &\leqslant  \vnorma{\rho}_{{\mathcal X}} \vnorma{\chi} \int_0^\eta \vnorma{\int_t^{\lceil t\rceil_{{\tau_{n_l}}}} \frac{\overline{d}_{n_l}(\lceil t\rceil_{\tau_{n_l}},r-\tau_{n_l})^{-\overline{B}_{n_l}(t)}}{\Gamma\left(1-\overline{B}_{n_l}(t)\right)} \partial_r U_{n_l}(r)\dr} \dt
\end{align*}
Using $\max_{t\in [0,T]} \vnorma{\pdt {U}_n(t)}^2 \leqslant {\mathcal C},$ we have that 
\begin{align*}
  &  \vnorma{\int_t^{\lceil t\rceil_{{\tau_{n_l}}}} \frac{ \overline{d}_{n_l}(\lceil t\rceil_{\tau_{n_l}},r-\tau_{n_l})^{-\overline{B}_{n_l}(t)}}{\Gamma\left(1-\overline{B}_{n_l}(t)\right)} \partial_r U_{n_l}(r)\dr}^2  \\
  &  \leqslant \int_\Theta \left( \int_t^{\lceil t\rceil_{{\tau_{n_l}}}} \overline{d}_{n_l}(\lceil t\rceil_{\tau_{n_l}},r-\tau_{n_l})^{-\overline{B}_{n_l}(t)} \dr \right) \\
  & \qquad \qquad \times \left( \int_t^{\lceil t\rceil_{{\tau_{n_l}}}} \overline{d}_{n_l}(\lceil t\rceil_{\tau_{n_l}},r-\tau_{n_l})^{-\overline{B}_{n_l}(t)} \abs{ \partial_r U_{n_l}(\X,r)}^2 \dr \right) \dX \\
  & \leqslant {\mathcal C} \left( \int_t^{\lceil t\rceil_{{\tau_{n_l}}}} \overline{d}_{n_l}(\lceil t\rceil_{\tau_{n_l}},r-\tau_{n_l})^{-\overline{B}_{n_l}(t)} \dr \right)^2. \\
\end{align*}
As $\max_{t\in[0,T]} \abs{\overline{B}_{n_l}(t)} \leqslant \bar \mu < 1$, the following holds $ \forall\, t\in[0,T]:$
\[
\int_t^{\lceil t\rceil_{{\tau_{n_l}}}} (\lceil t\rceil_{{\tau_{n_l}}}+{\tau_{n_l}}-r)^{-\overline{B}_{n_l}(t)} \dr \leqslant \int_t^{\lceil t\rceil_{{\tau_{n_l}}}} {\tau_{n_l}}^{-\overline{B}_{n_l}(t)} \dr \leqslant {\tau_{n_l}}^{1  -\overline{B}_{n_l}(t) \pm \bar {\mu}} \leqslant \tau_{n_l}^{1-\bar{\mu}}.
\]
Hence, we finally obtain that 
\[
\abs{{\mathcal T}_1} \leqslant \vnorma{\rho}_{{\mathcal X}} \vnorma{\chi} \sqrt{{\mathcal C}} T \tau_{n_l}^{1-\bar{\mu}} \to 0 \quad \text{ as } l \to \infty. 
\]
Similarly, we deduce for ${\mathcal T}_2$ that 
\[
\abs{{\mathcal T}_2} \leqslant  \vnorma{\rho}_{{\mathcal X}} \vnorma{\chi} \sqrt{{\mathcal C}} \int_0^T \int_0^t \abs{ \frac{ \overline{d}_{n_l}(\lceil t\rceil_{\tau_{n_l}},r-\tau_{n_l})^{-\overline{B}_{n_l}(t)}}{\Gamma\left(1-\overline{B}_{n_l}(t)\right)} -  \frac{ (t-r)^{-\mu(t)}}{\Gamma\left(1-\mu(t)\right)}} \dr \dt. 
\]
For $r\in (0,t)$ it holds that
\begin{align*}
\abs{\frac{\overline{d}_{n_l}(\lceil t\rceil_{\tau_{n_l}},r-\tau_{n_l})^{-\overline{B}_{n_l}(t)}}{\Gamma\left(1-\overline{B}_{n_l}(t)\right)}  -  \frac{ (t-r)^{-\mu(t)}}{\Gamma\left(1-\mu(t)\right)} } 
& \leqslant (t-r)^{-\overline{B}_{n_l}(t)} + (t-r)^{-\mu(t)} \\ & \stackrel{\eqref{eq:useful_estimate}}{\leqslant}
2\max\{1,T\} (t-r)^{-\bar{\mu}},
\end{align*}
Hence, the Lebesgue dominated theorem can be invoked here to get for almost all $t\in \I$ that 
\[
 \int_0^t \abs{ \frac{ \overline{d}_{n_l}(\lceil t\rceil_{\tau_{n_l}},r-\tau_{n_l})^{-\overline{B}_{n_l}(t)}}{\Gamma\left(1-\overline{B}_{n_l}(t)\right)} -  \frac{ (t-r)^{-\mu(t)}}{\Gamma\left(1-\mu(t)\right)}} \dr \to 0 \quad \text{ as } l \to \infty,
\]
since $\overline{B}_{n_l}(t) \to \mu(t)$ for almost all $t\in \I$ as $l\to \infty$.
Therefore, we also have that
\[
\abs{{\mathcal T}_2} \to 0 \quad \text{ as } \quad l \to \infty.
\]
For the next limit transition,  we first notice that 
\begin{align*}
\abs{{\mathcal T}_3} &\leqslant  \vnorma{\chi} \int_0^\eta \vnorma{\overline{P}_{n_l}(t)-\rho(t)}_{\lp{\infty}} \vnorma{\int_0^{t} \frac{ (t-r)^{-\mu(t)}}{\Gamma\left(1-\mu(t)\right)} \partial_r U_{n_l}(r)\dr}\dt \\
& \stackrel{\eqref{eq:useful_estimate}}{\leqslant} \vnorma{\chi}  \sqrt{\mathcal C} \max\{1,T\} T^{1-\bar\mu}   \int_0^\eta \vnorma{\overline{P}_{n_l}(t)-\rho(t)}_{\lp{\infty}} \dt. 
\end{align*}
Then, it follows from the Lebesgue-dominated theorem that 
\[
\abs{{\mathcal T}_3} \to 0 \quad \text{ as } \quad l \to \infty,
\]
since  $\vnorma{\overline{P}_{n_l}(t)-\rho(t)}$ for almost all $t\in \I$ as $l\to \infty.$
Now, for $z\in \lpkIX{2}{\lp{2}}$, we have that 
\begin{align*}
& \abs{\int_0^{\eta} \scal{\rho(t) \int_0^{t} \frac{ 1}{\Gamma\left(1-\mu(t)\right)(t-r)^{\mu(t)}} z(r)\dr }{\chi} \dt} \\
& \leqslant \vnorma{\chi} \int_0^{\eta} \vnorma{ \rho(t) \partial^{\mu(t)}_t z(t)} \dt \\
& \leqslant  \vnorma{\chi}  \sqrt{T} \vnorma{\rho \partial^{\mu(t)}_t z}_{\lpkIX{2}{\lp{2}}} \\
& \stackrel{\eqref{eq:useful_estimate2}}{\leqslant } \vnorma{\chi} \max\{1,T\} T^{\frac{3}{2}-\bar \mu}   \vnorma{\rho}_{{\mathcal X}}  \vnorma{z}_{\lpkIX{2}{\lp{2}}}.
\end{align*}
Hence, the weak convergence of $\pdt {U}_{n_l}$ to $\pdt \Phi$ in $\lpkIX{2}{\lp{2}}$ implies that
\[
\abs{{\mathcal T}_4} \to 0 \quad \text{ as } \quad l \to \infty.
\]
Therefore, if we pass to the limit $l\to \infty$ in \eqref{eq:varfor_whole_time_frame_int}, then we obtain that 
 \begin{multline*}
\int_0^\eta \langle \pdtt \Phi(t), \chi \rangle \dt + \int_0^\eta  \scal{\rho(t) \partial^{\mu(t)}_t \Phi(t)}{\chi} \dt +  \int_0^\eta  \scal{\beta(t)\nabla \Phi(t)}{\nabla \chi} \dt \\
 = \int_0^\eta \scal{f\left(\Phi(t)\right)}{\chi} \dt + \int_0^\eta  \scal{Q(t)}{\chi} \dt. 
 \end{multline*}
Finally, differentiating the result with respect to $\eta$ concludes the proof.
\end{proof}

From Theorem~\ref{thm:uniqueness1} it follows that the solution is unique if $\pdt \Phi \in \lpkIX{2}{\hko{1}}.$

%%%%%%%%%%%%%%%%%%%%%%%%%%%%%%%%
%%%%%%%%%%%%%%%%%%%%%%%%%%%%%%
\section{Numerical implementation}
\label{sec: numerical}
%%%%%%%%%%%%%%%%%%%%%%%%%%%%%%%%%%
%%%%%%%%%%%%%%%%%%%%%%%%%%%%%

\subsection{Numerical scheme}
We  discuss the numerical implementation of \eqref{eq:var_for_discrete}. In our implementation, we consider $\rho(t)$ and $\beta(t)$ as functions in time. We use the Legendre basis functions
$${\mathcal W}^{\mathcal N}_{0}  = \text{span}\left\{ {\chi _r (x):r = 0,1 \ldots ,\mathcal{N} - 2} \right\},$$
in which $\chi _r (x)$ is defined on $[a,b]$ by
\begin{equation*}
  \chi _r (x) = L _r (\hat x) - L _{r + 2} (\hat x) = \frac {2 r + 3} {2 (r + 1)} (1 - \hat x ^2) P_r ^{1 , 1} (\hat x), \quad \hat x := \frac{2x - b-a}{b-a} \in [-1 , 1],
\end{equation*}
where $L _r (\hat x)$ is the Legendre polynomial. It is obvious that ${\mathcal W}^{\mathcal N}_{0}$ is a subspace of $\Hi^1_0(\Theta)$. The numerical solution $\Phi_N \in {\mathcal W}^{\mathcal N}_{0}$ can be given by
\begin{equation*}
  \Phi ^\mathcal{N} _i = \sum _{\ell = 0} ^{\mathcal{N} - 2} {\hat \Phi _\ell ^i \chi _\ell (x)}.
\end{equation*} 
We consider the orthogonal projection operators $\pi_{\mathcal N}^{0}$ and $\pi_{\mathcal N}^{1,0}$ in the following manner 
\begin{itemize}
    \item $\pi_{\mathcal N}^{0}:\Leb^2(\Theta)\rightarrow \mathcal {W}^{\mathcal N}_{0}$  such that if $\Phi \in \Leb^2(\Theta)$ then $\pi_{\mathcal N}^{0}\Phi\in \mathcal {W}^{\mathcal N}_{0}$ fulfils 
\begin{equation*}
(\pi_{\mathcal N}^{0}\Phi, \chi)=(\Phi,\chi),\quad \forall \chi \in \mathcal {W}^{\mathcal N}_{0};
\end{equation*}
\item $\pi_{\mathcal N}^{1,0}:\Hi_{0}^{1}(\Theta)\rightarrow \mathcal {W}^{\mathcal N}_{0}$ such that if $ \Phi\in \Hi_{0}^{1}(\Theta)$ then $\pi_{\mathcal N}^{1,0}\Phi \in \mathcal {W}^{\mathcal N}_{0}$ satisfies
\begin{equation*}
(\partial_{x}\pi_{\mathcal N}^{1,0}\Phi, \partial_{x} \chi)=(\partial_{x}\Phi,\partial_{x}\chi),\quad \forall  \chi \in \mathcal {W}^{\mathcal N}_{0}.
\end{equation*}
\end{itemize}
Then, the fully discrete  scheme of \eqref{eq:var_for_discrete} can be expressed as follows: for $i= 1,\,2,\ldots, n,$ find $\Phi ^\mathcal{N} _i  \in {\mathcal W}^{\mathcal N}_{0}$ such that
\begin{align}\label{deq88var_for_discrete2}
\begin{split}
a_i(\Phi^\mathcal{N}_i,\chi)=\langle l_i,\chi\rangle, \quad \forall \chi \in {\mathcal W}^{\mathcal N}_{0},\\
\Phi^\mathcal{N}_0 = \pi_{N}^{1,0} \tilde \Phi_0, \\
\delta\Phi^\mathcal{N}_0 = \pi_{N}^{0} \tilde \Psi_0,
\end{split}
\end{align}
where
$$
a_i(\Phi^\mathcal{N}_i,\chi) := \tau^{-2} \scal{\Phi^\mathcal{N}_i}{\chi} + \frac{\tau^{-\mu_i}}{\Gamma(1-\mu_i)} \scal{\rho_i \Phi^\mathcal{N}_i}{\chi}  +  \scal{\beta_i \nabla \Phi^\mathcal{N}_i}{\nabla \chi}
$$
and 
\begin{multline*}
\langle l_i,\chi\rangle := \scal{\pi_{\mathcal N}^{0} f\left(\Phi^\mathcal{N}_{i-1}\right)}{\chi} + \scal{\pi_{\mathcal N}^{0} Q_i}{\chi} + \tau^{-2} \scal{\Phi^\mathcal{N}_{i-1}}{\chi} + \tau^{-1} \scal{\delta \Phi^\mathcal{N}_{i-1}}{\chi} \\
+ \frac{\tau^{-\mu_i}}{\Gamma(1-\mu_i)} \scal{\rho_i \Phi^\mathcal{N}_{i-1}}{\chi} 
- \sum_{l=1}^{i-1}  \frac{(t_i-t_{l-1})^{-\mu_i}}{\Gamma\left(1-\mu_i\right)} \scal{\rho_i\delta \Phi^\mathcal{N}_l}{\chi} \tau. 
\end{multline*}
Let us denote
\begin{multline*}
d_i  = \frac{1 }{\tau^2 } +\rho_i\, b_i^i,\quad \hat b_{q}^{i}  = -b_{q}^i \rho_i + \frac{1 }{\tau^2 }\delta _{q,i - 1},\\ \hat Q_i= \pi_{\mathcal N}^{0} Q_i +\frac{1}{\tau}V_i+ \pi_{\mathcal N}^{0} f(\Phi_{i-1}^\mathcal{N}),\, V_i= \delta \Phi_{i-1}^\mathcal{N}.
\end{multline*}
Define the matrices $\bar{M},\, S,\, L,$ which satisfy
\begin{equation*}\label{mzaky}
\begin{split}
\left( {S  } \right) _{i , j = 0} ^{\mathcal{N} - 2} =    S _{ij} &= \displaystyle \int _\Theta  \chi'_i (x) \chi' _j (x) \dx ,  \\
  \left(  \bar{M}  \right) _{i , j = 0} ^{\mathcal{N} - 2}=   m _{ij} &= \displaystyle \int _\Theta \chi _i (x) \chi _j (x) \dx,\\
    \left( {L  } \right) _{i , j = 0} ^{\mathcal{N} - 2} =   L _{ij} &= \displaystyle \int _\Theta \chi _j (x) \hat Q _i (x) \dx,\\
        U_i &= \displaystyle (\hat \Phi _0 ^i, \hat \Phi _1 ^i , \ldots , \hat \Phi _{N - 2} ^i) ^\top. \\ 
\end{split}
\end{equation*}
The fully discrete scheme \eqref{deq88var_for_discrete2} can be written in the matrix form as
\begin{equation*}\label{matd5t5t}
  \left( d_i\,\bar{M} + \beta_i  S\right) U_i = L+ \sum _{q = 0} ^{i - 1} {\hat b^{i}_{q} \bar{M} U _q}. 
\end{equation*}
The stiffness and mass matrix are specified in the following lemma \cite{shen2011spectral}. 

\begin{lemma} 
  The stiffness matrix $S$ is a diagonal matrix with
  \begin{equation*}\label{lem11a}
  s_{dd}  = 4d+6,\ d=0,1,\ldots.
  \end{equation*}
  The mass matrix $\bar{M}$ is symmetric with the nonzero elements 
  \begin{equation*}\label{lem11ggb}
\begin{split}
m_{dj}  = m_{jd}  = \left\{ \begin{array}{l}
 \frac{b-a}{{2j + 1}} + \frac{b-a}{{2j + 5}},\quad d = j, \\
 \\
  - \frac{b-a}{{2j + 5}},\quad\quad d = j + 2. \\
 \end{array} \right.\end{split}
  \end{equation*}
\end{lemma}

 %%%%%%%%%%%%%%%%%%%%%%%%%%%%%%%%%%%%%%%
 %%%%%%%%%%%%%%%%%%%%%%%%%%%%%%%%%%%%

\subsection{Numerical examples}

We introduce the rate of convergence in $\Leb^2$-norm  as
\begin{equation*}
   \text{C-O} = \frac{\abs{\log \left( {{E\left( {\mathcal N,M_1 } \right)} /{E\left( {\mathcal N ,M_2 } \right)} } \right)}}{\abs{\log \left( {M_2 /M_1 } \right)}},
\end{equation*}
 and the  approximation  order in the spatial direction  (A-O)  by:
\begin{equation*}
   \text{ A-O} = \frac{{\log \left( {{E\left( {\mathcal N,M } \right)} } \right)}}{{\log \left( {{\mathcal N} } \right)}},
\end{equation*}
where $M_1 \neq M_2$ and $
  E = E(\mathcal N,M)= \mathop {\max }\limits_{1 \leqslant i \leqslant M} \left\| {\Phi^{\mathcal N}_i  - \Phi_i } \right\| .
$

\begin{example}\label{exa1} Consider the following  nonlinear problem
  \begin{align*}
  \label{ex1}%
       \frac {\partial^2  \Phi} {\partial t^2} +\frac {\partial ^{\mu(t)} \Phi} {\partial t ^{\mu(t)}} &=  \frac {\partial ^2 \Phi} {\partial x^2} +\Phi(1-\Phi)+g,\quad x \in (0,1),\quad t \in (0,1],\\
     &\Phi(0,t)=\Phi(1,t)=0, \quad t \in (0,1),\\
     & \Phi (\X , 0)= \pdt  \Phi (\X , 0)=0,\quad x\in  \Theta,
\end{align*}
where $g$ is given such this problem has the exact solution $\Phi(x,t)=t^2 \sin(\pi\, x).$
The V-OF $\mu (t)$ is given by
\begin{equation*}
    \mu (t) = \mu (T) + \left( {\mu (0) - \mu (T)} \right)\left( {1 - \frac{t}{T} - \frac{{\sin \left( {2\pi \left( {1 - \frac{t}{T}} \right)} \right)}}{{2\pi }}} \right).\end{equation*}
The behavior of the numerical solution with   $\mu (0)=0.2$ and $\mu (1)=0.4$  is investigated.
 Table~\ref{Table1} provides the $\Leb^2$-errors and related convergence order with $\mathcal N = 50$. It also provides the rate of convergence in the spatial direction  of various values of ${\mathcal N}$  with $\tau=0.000625$. We observe  that   the convergence order in the time direction is about $(1-\max(\mu))$. \\
\end{example}

\begin{table}
                       \centering
                     \caption{Example \ref{exa1}: The maximum errors and the rate of convergence  versus  $\mathcal N$ and $\tau$. \label{Table1} }
                  \centerline{ \begin{tabular}{ccc|ccccccccccc}
                   \hline
        $\tau$                               & E  &C-O&${\mathcal N}$&E&A-O \\
                    \hline
0.01    &$3.280 \times 10^{-3}$&$--$   &  $5$    &$6.650\times10^{-4}$ &${\mathcal N}^{-4.545}$     \\
0.005    &$2.245 \times 10^{-3}$&$0.547$  &    $10$  &$6.705\times10^{-4}$&${\mathcal N}^{-3.173}$     \\
0.0025    &$1.516 \times 10^{-3}$&$0.565$  &   $20$     &$6.705\times10^{-4}$&${\mathcal N}^{-2.439}$     \\
0.00125    &$1.013 \times 10^{-3}$&$0.582$  &  $30$  &$6.705\times10^{-4}$&${\mathcal N}^{-2.148}$     \\
0.000625   &$6.705 \times 10^{-4}$&$0.595$    &  $40$     &$6.705\times10^{-4}$&${\mathcal N}^{-1.980}$     \\
$\tau^{1 - \max \left( \mu  \right)} $    &$--$&$0.600$    &   $50$  &$6.705\times10^{-4}$&${\mathcal N}^{-1.867}$     \\
                    \hline
                  \end{tabular}}
                \end{table}

\begin{example}\label{exa2} We consider the V-OF wave equation:
  \begin{equation}
  \label{ex2}%
 \frac {\partial ^{\mu(t)} \Phi} {\partial t ^{\mu(t)}} + \frac {\partial^2  \Phi} {\partial t^2} =  \frac {\partial ^2 \Phi} {\partial x^2}   + \Phi(1+\Phi)+g,
     \quad x \in (0,1),\quad t \in (0,1],
\end{equation}
where $g(x,t)$ and the initial-boundary conditions are  given    such that problem \eqref{ex2} has the exact solution $\Phi(x,t)=t^2 x^2(1-x)^2.$
In this example, the V-OF $\mu (t)$ is given by
\begin{enumerate}
    \item [(I)]  Linear $\mu (t)$: \\
$\mu (t)= \mu(T)+ (\mu(0) - \mu(T))(1 - t),$\\
$\mu(0) = 0.6,$ $\mu(T) = 0.4$.
        \item [(II)] Quadratic $\mu(t)$:\\
$\mu(t)= \mu(T) + (\mu(0) - \mu(T))(1 - t^2),$\\
$\mu(0) := 0.5,$ 
$\mu(T) := 0.8.$
            \item [(III)] Osciliating $\mu(t)$:\\
$\mu (t)=\mu(T) + (\mu(0) - \mu(T))(1 - \frac{1}{2\, \pi}\sin(2\, \pi\, (1 - t))),$\\
$\mu(0) = 0.6,$
$\mu(T) := 0.8.$
\end{enumerate}
 For above three cases, Tables~\ref{TA2}, \ref{TA3}, \ref{TA4} provides the $\Leb^2$-errors and related convergence order with $\mathcal N = 50$. It also provides the rate of convergence in the spatial direction of various values of ${\mathcal N}$  with $\tau=0.000625$. We observe numerically that the convergence order in the time direction is about $(1-\max(\mu))$, and the accuracy of the spectral approximation is tied to the rate of convergence in the time direction.\\
\end{example}

\begin{table}
                       \centering
                     \caption{Example \ref{exa2}: The maximum errors and the rate of convergence  versus  $\mathcal N$ and $\tau$ for case $(I)$.\label{TA2} }
                  \centerline{ \begin{tabular}{ccc|ccccccccccc}
                   \hline
        $\tau$                               & E  &C-O&${\mathcal N}$&E&A-O \\
                    \hline
0.01    &$4.416 \times 10^{-4}$&$--$   &  $5$    &$1.314\times10^{-4}$ &${\mathcal N}^{-5.552}$     \\
0.005    &$3365 \times 10^{-4}$&$0.3923$  &    $10$  &$1.314\times10^{-4}$&${\mathcal N}^{-3.881}$     \\
0.0025    &$2.497 \times 10^{-4}$&$0.430$  &   $20$     &$1.314\times10^{-4}$&${\mathcal N}^{-2.983}$     \\
0.00125    &$1.822 \times 10^{-4}$&$0.454$  &  $30$  &$1.314\times10^{-4}$&${\mathcal N}^{-2.627}$     \\
0.000625   &$1.314 \times 10^{-4}$&$0.470$    &  $40$     &$1.314\times10^{-4}$&${\mathcal N}^{-2.422}$     \\
$\tau^{1 - \max \left( \mu  \right)} $    &$--$&$0.400$    &   $50$  &$1.314\times10^{-4}$&${\mathcal N}^{-2.284}$     \\
                    \hline
                  \end{tabular}}
                \end{table}

\begin{table}
                       \centering
                     \caption{Example \ref{exa2}:  The maximum errors and the rate of convergence  versus  $\mathcal N$ and $\tau$  for case $(II)$.\label{TA3} }
                  \centerline{ \begin{tabular}{ccc|ccccccccccc}
                   \hline
        $\tau$                               & E  &C-O&${\mathcal N}$&E&A-O \\
                    \hline
0.01    &$1.153 \times 10^{-3}$&$--$   &  $5$    &$4.634\times10^{-4}$ &${\mathcal N}^{-4.770}$     \\
0.005    &$9.243 \times 10^{-4}$&$0.319$  &    $10$  &$4.6342\times10^{-4}$&${\mathcal N}^{-3.334}$     \\
0.0025    &$7.362 \times 10^{-4}$&$0.328$  &   $20$     &$4.6342\times10^{-4}$&${\mathcal N}^{-2.562}$     \\
0.00125    &$5.845 \times 10^{-4}$&$0.333$  &  $30$  &$4.6342\times10^{-4}$&${\mathcal N}^{-2.257}$     \\
0.000625   &$4.634 \times 10^{-4}$&$0.334$    &  $40$     &$4.6342\times10^{-4}$&${\mathcal N}^{-2.081}$     \\
$\tau^{1 - \max \left( \mu  \right)} $    &$--$&$0.2000$    &   $50$  &$4.6342\times10^{-4}$&${\mathcal N}^{-1.962}$     \\
                    \hline
                  \end{tabular}}
                \end{table}

\begin{table}
                       \centering
                     \caption{Example \ref{exa2}: The errors and  convergence orders versus $\tau$ and $\mathcal N$ for case $(III)$.\label{TA4} }
                  \centerline{ \begin{tabular}{ccc|ccccccccccc}
                   \hline
        $\tau$                               & E  &C-O&${\mathcal N}$&E&A-O \\
                    \hline
0.01    &$1.000 \times 10^{-3}$&$--$   &  $5$    &$4.634\times10^{-4}$ &${\mathcal N}^{-4.909}$     \\
0.005    &$7.924 \times 10^{-4}$&$0.335$  &    $10$  &$4.6342\times10^{-4}$&${\mathcal N}^{-3.430}$     \\
0.0025    &$6.196 \times 10^{-4}$&$0.354$  &   $20$     &$4.6342\times10^{-4}$&${\mathcal N}^{-2.637}$     \\
0.00125    &$4.805 \times 10^{-4}$&$0.366$  &  $30$  &$4.6342\times10^{-4}$&${\mathcal N}^{-2.141}$     \\
0.000625   &$3.707 \times 10^{-4}$&$0.374$    &  $40$     &$4.6342\times10^{-4}$&${\mathcal N}^{-2.019}$     \\
$\tau^{1 - \max \left( \mu  \right)} $    &$--$&$0.368$    &   $50$  &$4.6342\times10^{-4}$&${\mathcal N}^{-1.962}$     \\
                    \hline
                  \end{tabular}}
                \end{table}

\begin{example}\label{exa3} We consider the following  nonlinear  differential equation with variable coefficients
  \begin{equation}
  \label{ex3}%
 e^{-t} \frac {\partial ^{\gamma(t)} \Phi} {\partial t ^{\gamma(t)}} + \frac {\partial^2  \Phi} {\partial t^2} =  \frac {\partial ^2 \Phi} {\partial x^2}  + \Phi(1+\Phi)+g,
     \quad x \in (0,1),\quad t \in (0,1],
\end{equation}
where $g(x,t)$ and the initial-boundary conditions are  given    such that problem \eqref{ex2} has the exact solution $\Phi(x,t)=t^2 \sin(\pi x).$
In this example, the V-OF $\mu (t)$ is given by
\[
\mu (t) = \left\{ \begin{array}{l}
 \frac{1}{4},\quad 0 < t \le \frac{1}{2} \\
 \\
 \frac{3}{4},\quad 
 \frac{1}{2} < t \le 1. \\ 
 \end{array} \right.
\]
\end{example}

\begin{table}
                       \centering
                     \caption{Example \ref{exa3}:  The maximum errors and the rate of convergence  versus  $\mathcal N$ and $\tau$.\label{TA5} }
                  \centerline{ \begin{tabular}{ccc|ccccccccccc}
                   \hline
        $\tau$                               & E  &C-O&${\mathcal N}$&E&A-O \\
                    \hline
0.01    &$6.113 \times 10^{-2}$&$--$   &  $5$    &$2.977\times10^{-2}$ &${\mathcal N}^{-2.183}$     \\
0.005    &$5.114 \times 10^{-2}$&$0.257$  &    $10$  &$2.979\times10^{-2}$&${\mathcal N}^{-1.525}$     \\
0.0025    &$4.272 \times 10^{-2}$&$0.259$  &   $20$     &$2.979\times10^{-2}$&${\mathcal N}^{-1.172}$     \\
0.00125    &$3.567 \times 10^{-2}$&$0.260$  &  $30$  &$2.979\times10^{-2}$&${\mathcal N}^{-1.032}$     \\
0.000625   &$2.979 \times 10^{-2}$&$0.259$    &  $40$     &$2.979\times10^{-2}$&${\mathcal N}^{-0.952}$     \\
$\tau^{1 - \max \left( \mu  \right)} $    &$--$&$0.250$    &   $50$  &$2.979\times10^{-2}$&${\mathcal N}^{-0.898}$     \\
                    \hline
                  \end{tabular}}
                \end{table}
 Table~\ref{TA5} provides the $\Leb^2$ errors and corresponding convergence orders with $\mathcal N = 50$. It also provides the spatial convergence orders for various values of ${\mathcal N}$  with $\tau=0.000625$. We observe that the temporal convergence rate is $(1-\max(\mu))$.

%%%%%%%%%%%%%%%%%%%%%%%%
%%%%%%%%%%%%%%%%%%%%%%%
\section{Conclusions}

The wave equation with V-OF damping and variable coefficients has been considered. 
The well-posedness of this problem has been investigated. A weak formulation for the considered  problem has been proposed based on reasonable assumptions on the given data. The uniqueness of a solution has been established by employing Gr\"onwall's lemma, and the existence of the solution has been proved using a Rothe scheme. This time-discrete scheme is based on a discrete convolution approximation in the backward sense. Additionally, a full discretisation of the problem has been constructed using Galerkin spectral techniques in the spatial direction, and numerical examples have been tested. Our future plan will be devoted to stating and proving the error estimates for the full discretisation of the problem under consideration. Also, we will seek to construct the Rothe scheme over nonuniform meshes for the problem.

		\bibliography{ref}
		\bibliographystyle{abbrv}
\end{document}